\documentclass[11pt,letterpaper]{amsart}

\usepackage{amsmath, amssymb, amsfonts, stmaryrd,mathrsfs}
\usepackage{hyperref}

\DeclareMathOperator{\Hom}{Hom}

\DeclareMathOperator{\tr}{tr}

\DeclareMathOperator{\res}{res}

\DeclareMathOperator{\Idem}{Idem}
\DeclareMathOperator{\OP}{OP}
\newcommand{\wt}{\widetilde}
\newcommand{\M}{{\mathcal M}}
\newcommand{\IP}{{\mathcal {IP}}}

\newcommand{\C}{{\mathbb C}}
\newcommand{\N}{{\mathbb N}}

\newcommand{\dbar}{{\, {}^- \hspace{-3mm} d}}

\renewcommand{\Re}{{\operatorname {Re}}\,}

\theoremstyle{plain}
\newtheorem{prop}{Proposition}
\newtheorem{thm}[prop]{Theorem}
\newtheorem{lemma}[prop]{Lemma}

\theoremstyle{definition}

\theoremstyle{plain}

\numberwithin{equation}{section}

\title[Residue of projections in Boutet de Monvel's calculus]{Non-commutative residue of projections in Boutet de Monvel's calculus}
\author{Anders Gaarde}
\address{University of Copenhagen \\ Department of Mathematical
  Sciences \\ Universitetsparken 5 \\ DK-2100 Copenhagen, Denmark.}
\email{gaarde@math.ku.dk}

\begin{document}

\begin{abstract}
Using results by Melo, Nest, Schick, and Schrohe on the \linebreak
$K$-theory of 
Boutet de Monvel's calculus of boundary value problems, we show
that the non-commutative residue introduced by Fedosov, Golse,
Leichtnam, and Schrohe vanishes on projections in the calculus.

This partially answers a question raised in a recent collaboration with
Grubb, namely whether the residue is zero on sectorial projections for
boundary value problems: This is confirmed to be true when the
sectorial projections is in the calculus.
\end{abstract}

\maketitle

\section{Introduction}

Boutet de Monvel \cite{BdM71} constructed a calculus, often called the
Boutet de Monvel calculus (or algebra), of pseudodifferential boundary
operators on a manifold with boundary. It includes the classical
differential boundary value problems as well of the parametrices of
the elliptic elements:

Let $X$ be a compact $n$-dimensional manifold with boundary $\partial
X$; we consider $X$ as an embedded submanifold of a closed
$n$-dimensional manifold $\widetilde X$. Denote by $X^\circ$ the
interior of $X$. Let $E$ and $F$ be smooth complex vector bundles over
$X$ and $\partial X$, respectively, with $E$ the restriction to $X$ of a
bundle $\wt E$ over $\wt X$. 

An operator in Boutet de Monvel's calculus --- a (polyhomogeneous)
Green operator --- is a map $A$ acting on 
sections of $E$ and $F$, given by a matrix
\begin{equation}
  \label{eq:greenop0}
  A = \begin{pmatrix} P_+ + G & K \\[2mm] T & S \end{pmatrix} \ : \
  \begin{matrix} C^\infty(X,E) & & C^\infty(X,E) \\ \oplus & \to & \oplus
    \\ C^\infty(\partial X,F) & & C^\infty(\partial X,F) \end{matrix},
\end{equation}
where $P$ is a pseudodifferential operator ($\psi$do) on $\wt X$ with
the transmission property and $P_+$ is its truncation to $X$:
\begin{equation}
  \label{eq:trunc}
  P_+ = r^+ P e^+, \quad r^+ \text{ restricts from $\wt X$ to
    $X^\circ$, $e^+$ extends by 0.}
\end{equation}
$G$ is a singular Green operator, $T$ a trace operator, $K$ a Poisson
operator, and $S$ a $\psi$do on the closed manifold $\partial X$. See
\cite{BdM71}, Grubb \cite{G-book}, or Schrohe \cite{Sch01} for details.

Fedosov, Golse, Leichtnam, and Schrohe \cite{FGLS96} extended the
notion of noncommutative residue known from closed mani\-folds (cf.\
Wodzicki \cite{Wod82}, \cite{Wod84}, and Guillemin \cite{Gui85}) to
the algebra of Green operators. The 
noncommutative residue of $A$ from \eqref{eq:greenop0} was
defined to be
\begin{multline}
  \label{eq:fglsres}
  \res_X (A) = \int_X \int_{S_x^*X} \tr_E p_{-n}(x,\xi) \dbar S(\xi) dx
  \\ + \int_{\partial X}\int_{S_{x'}^* \partial X} \big[ \tr_E (\tr_n
  g)_{1-n}(x',\xi') + \tr_F s_{1-n}(x',\xi') \big] \dbar S(\xi') dx'.
\end{multline}
Here $\tr_E$ and $\tr_F$ are traces in $\Hom(E)$ and
$\Hom(F)$, respectively; $\dbar S(\xi)$ (resp.\ $\dbar
S(\xi')$) denotes the surface measure 
divided by
$(2\pi)^n$ (resp.\ $(2\pi)^{n-1}$); $\tr_n g$ is the normal trace of
$g$; and the subscripts $-n$ and $1-n$ indicate that we only
consider the homogeneous terms of degree $-n$ resp.\ $1-n$. Also, a
sign error in \cite{FGLS96} has been corrected, cf.\ Grubb and Schrohe
\cite[(1.5)]{GSc01}.

It is well-known \cite{Wod82} that on a closed manifold, the noncommutative
residue of a classical $\psi$do projection (or idempotent) is zero. In
the present paper we wish to show 
that the same holds in the case of Green operators. We will use
$K$-theoretic arguments (in a $C^*$-algebra setting) to effectively
reduce the problem to the known case of closed manifolds.

\medskip

In our recent collaboration with Grubb \cite{GG07} we studied certain
spectral projections: For the realization $B = (P+G)_T$ of an 
elliptic boundary value problem $\{ P_+ + G, T\}$ of order $m>0$ with
two spectral 
cuts at angles $\theta$ and $\varphi$, one can define the
\textit{sectorial projection} $\Pi_{\theta,\varphi}(B)$. It is a (not
necessarily self-adjoint) projection whose range contains the 
generalized eigenspace of $B$ for the sector $\Lambda_{\theta,\varphi}
= \{ r e^{i \omega} \mid r>0,\, \theta < \omega < \varphi \}$ and whose
nullspace contains the generalized eigenspace for $\Lambda_{\varphi,\theta+2\pi}$.
It was considered earlier by Burak \cite{Bu70}, and in the boundary-less
case by Wodzicki \cite{Wod82} and Ponge \cite{Po06}.

In general this operator is not in Boutet de Monvel's calculus, but we
showed that it has a residue in a slightly more general
sense. The question was posed whether this residue vanishes.

The question of the non-commutative residue of projections is particularly
interesting in the context of zeta-invariants as discussed by Grubb
\cite{G-zeta} and in \cite{GG07}: The \textit{basic zeta value}
$C_{0,\theta}(B)$ for the realization $B$ of a boundary value problem 
is defined via a choice of spectral cut in the complex plane; the
difference in the basic zeta value based on two spectral cut angles $\theta$
and $\varphi$ is then given as the non-commutative residue 
of the corresponding sectorial projection:
\begin{equation}
  \label{eq:zeta}
  C_{0,\theta}(B) - C_{0,\varphi}(B) = \frac{2 \pi i}{m} \res_X (
  \Pi_{\theta,\varphi}(B) ).
\end{equation}
Our results here show that the dependence of $C_{0,\theta}(B)$ upon
$\theta$ is trivial whenever the projection $\Pi_{\theta,\varphi}(B)$
lies in Boutet de Monvel's calculus. 

\medskip


It should be noted that the litterature in functional analysis and
PDE-theory often uses ``projection'' as a synonym for idempotent,
while $C^*$-algebraists furthermore require that projections are
self-adjoint; we will try to avoid confusion by explicitly using the 
term ``$\psi$do projection'' for the idempotent operators here.

\section{Preliminaries and notation}

We employ Blackadar's \cite{Bla98} approach to $K$-theory: A pre-$C^*$-algebra
$B$ is called local if it, as a subalgebra of its $C^*$-completion
$\overline B$, is closed under holomorphic function
calculus (and all of its matrix algebras must have this property as well). 
Let $\M_\infty(B)$ denote the direct
limit of the matrix algebras $\M_m(B)$, $m\in\N$. Define $\IP_\infty(B)=
\Idem(\M_\infty(B))$ --- resp.\ $\IP_m(B) = \Idem(\M_m(B))$ --- to be the
set of all --- resp.\ all $m\times m$ --- idempotent matrices with entries from
$B$. Define the relation $\sim$ on $\IP_\infty(B)$ by 
\begin{equation}
  \label{eq:eqrel}
  x \sim y \text{ if there exist } a,b \in \M_\infty(B) \text{ such
    that } x = a b \text{ and } y = b a.
\end{equation}
If $B$ has a unit we define $K_0(B)$ to be the Grothendieck group of the
semigroup $V(B) = \IP_\infty(B) / \sim$. If $B$ has no unit, we consider the
scalar map from the unitization --- indicated with a tilde as in $\wt
B$ or $B^\sim$ --- of $B$ to the complex numbers
$s : \wt B \to \C$ defined by $s(b + \lambda 1_{\wt B}) = \lambda$,
and then define $K_0(B)$ as the kernel of the induced map $s_* :
K_0(\wt B) \to K_0(\C)$.

A fact that we shall use several times is that if $B$ is local, then
\cite[p. 28]{Bla98}
\begin{equation}
  \label{eq:local}
  V(B) \cong V(\overline B), \text{ and hence } K_0(B) \cong
  K_0(\overline B).
\end{equation}
Combined with \textit{the standard picture of $K_0$} this
implies that 
\begin{equation}
  \label{eq:stpic0}
  K_0(\overline B) = \{ \, [x]_0 - [y]_0 \mid x,y \in \IP_m(B), m\in
  \N \, \}
\end{equation}
in the case where $B$ is unital, and
\begin{equation}
  \label{eq:stpic1}
  K_0(\overline B) = \{ \, [x]_0 - [y]_0 \mid x,y \in \IP_m(\wt B)
  \text{ with } x \equiv y \text{ mod } \M_m(B), m\in \N \, \}
\end{equation}
in the non-unital case \cite{Bla98}.

\medskip

Let $\mathscr A$ denote the set of Green operators as in
\eqref{eq:greenop0} of order and class zero; it 
defines a $*$-subalgebra of the bounded operators on the Hilbert space
$\mathcal H = L_2(X,E)\oplus H^{-\frac 12}(\partial X,F)$; 
we will denote by $\mathfrak A$ its $C^*$-closure in $\mathcal B(\mathcal
H)$. $\mathscr
A$ is local with $\overline{\mathscr A} = \mathfrak A$, cf.\
Melo, Nest, and Schrohe \cite{MNS03b}, so
$K_0(\mathscr A) \cong K_0(\mathfrak A)$. 
Note that the $K$-theory of $\mathscr A$ is independent of the
specific bundles \cite[Section 1.5]{MNS03b}, so for simplicity we
study explicitly in this paper only the simplest trivial case $E = X \times
\C$ and $F = \partial X \times \C$. 

$\mathscr K$ denotes the subalgebra of smoothing operators, $\mathfrak
K$ its $C^*$-closure (the ideal of compact operators).
We let $\mathscr I$ denote the set of elements in $\mathscr A$ of the
form 
\begin{equation}
  \label{eq:Iideal}
  \begin{pmatrix} \varphi P \psi + G & K \\ T & S \end{pmatrix}
\end{equation}
with $\varphi,\psi\in C_c^\infty(X^\circ)$, $P$ a $\psi$do on $\wt X$
of order zero, and $G, K, T$, and $S$ of negative order and class
zero. $\mathfrak I$ will be the $C^*$-closure of $\mathscr I$ in
$\mathfrak A$. 

The noncommutative residue defined in \cite{FGLS96} is a trace --- 
a linear map that vanishes on commutators --- $\res :
\mathscr A \to \C$, and therefore induces a group homo\-morphism $\res_*
: K_0(\mathscr A) \to \C$ such that
\begin{equation}
  \label{eq:res1}
  \res_*( [ A ]_0 ) = \res_X (A)
\end{equation}
for any idempotent $A \in \mathscr A$. Our goal is to
prove the vanishing of $\res_*$, which obviously implies that
$\res_X(A)=0$ for any idempotent $A$.

The quotient map $q : \mathfrak A \to \mathfrak A/\mathfrak K$ induces
an isomorphism $q_* : K_0(\mathfrak A) \to K_0(\mathfrak A/\mathfrak
K)$ \cite{MNS03b}. The
isomorphisms $K_0(\mathscr A) \cong K_0( \mathfrak A) \cong
K_0(\mathfrak A/\mathfrak K)$ allow us to 
extend the noncommutative residue: For each $[ \mathcal A+ \mathfrak K]_0$ in
$K_0(\mathfrak A/\mathfrak K)$ there is an $A \in \IP_\infty(\mathscr
A)$ such that $q_* [A]_0 = [ \mathcal A + \mathfrak K]_0$, and we then
define 
\begin{equation}
  \label{eq:restilde}
  \wt \res_* [\mathcal A + \mathfrak K]_0 = \res_* [A]_0 = \res_X (A).
\end{equation}
The map $\wt\res_*$ is really just $\res_* \circ \
q_*^{-1}$, and is thus a group homomorphism $K_0(\mathfrak A/\mathfrak K)
\to \C$. 


\section{K-theory and the residue}

We employ results from Melo, Schick, and Schrohe \cite{MSS06}, in
particular the fact that ``each element in $K_0(\mathfrak A/\mathfrak
K)$ can be written as the sum of two elements, 
one in the range of $m_*$ and one in the range of $s''$, thus in the
range of $i_*$'' (bottom of page 11). In other words
\begin{equation}
  \label{eq:split1}
K_0(\mathfrak A/\mathfrak K) = q_* m_* K_0( C(X) ) + i_* K_0 (
\mathfrak I/\mathfrak K).
\end{equation}
Here $m : C(X) \to \mathfrak A$ sends $f$ to the multiplication operator
$\begin{pmatrix} f & 0 \\ 0 & 0 \end{pmatrix}$ and $i$ is the
inclusion $\mathfrak I/\mathfrak K \to \mathfrak A/\mathfrak K$; $m_*$
and $i_*$ are then the corresponding induced maps in $K_0$. We will in
general suppress $i$ and $i_*$ to simplify notation.

We will show that $\wt\res_*$ vanishes on both terms in the right
hand side of \eqref{eq:split1}. The following lemma treats the first
of these terms: 

\begin{lemma} \label{lem:resmK}
$\widetilde\res_*$ vanishes on $q_* m_* K_0(C(X))$.
\end{lemma}

\begin{proof}
Recall that a multiplication operator is in particular a Green
operator whose noncommutative residue is zero.

Let $f\in \IP_m(C^\infty(X))$; $m(f)$ acts by multiplication with a
smooth (matrix) function
and therefore lies in $\IP_m(\mathscr A)$. Then
$q_* m_* [f]_0 = q_* 
[m(f)]_0 = [m(f) + \mathfrak K]_0$, and according to \eqref{eq:restilde}
\begin{equation}
  \label{eq:restilde2}
  \wt\res_* ( q_* m_* [f]_0 ) = \res_* [ m(f) ]_0 = \res_X (m(f)) = 0.
\end{equation}
Since $C^\infty(X)$ is local in $C(X)$ \cite[3.1.1-2]{Bla98}, any
element of $K_0(C(X))$ can be written as $[f]_0 - [g]_0$ for some $f,g
\in \IP_m(C^\infty(X))$, cf.\ \eqref{eq:stpic0}. The lemma follows from this.
\end{proof}

We now turn to the second term of \eqref{eq:split1}; our strategy is
to show that the elements of $K_0(\mathfrak I/\mathfrak K)$ correspond to
$\psi$dos with symbols supported in the interior of $X$. This allows
us to construct certain projections for which the noncommutative
residue is given as the residue of a projection on the closed manifold
$\widetilde X$.

The principal symbol induces an isomorphism $\mathfrak I/\mathfrak K
\cong C_0(S^*X^\circ)$ 
\cite[Theorem 1]{MNS03b}. We will denote the induced isomorphism in
$K_0$ by $\sigma_*$, i.e.,
\begin{equation}
  \label{eq:sigmaiso}
  \sigma_* : K_0( \mathfrak I/\mathfrak K)
  \stackrel{\cong}{\longrightarrow} K_0 (C_0 (S^*X^\circ)).
\end{equation} 

Like in Lemma \ref{lem:resmK} we wish to consider smooth functions
instead of merely continuous functions; the following shows that
instead of $C_0(S^*X^\circ)$, it 
suffices to look at smooth functions (symbols) compactly supported in
the interior:

The algebra $C_c^\infty(S^* X^\circ)$, equipped with the sup-norm,
is a local $C^*$-algebra \cite[3.1.1-2]{Bla98} with completion
$C_0(S^* X^\circ)$. It follows from \eqref{eq:local}
that the injection $C_c^\infty(S^* X^\circ) \to C_0(S^*X^\circ)$
induces an isomorphism 
\begin{equation} 
\label{eq:injiso} 
 K_0(C_c^\infty(S^* X^\circ)) \cong K_0(C_0(S^*X^\circ)).
\end{equation}

We now show that each compactly supported symbol in
$K_0(C_c^\infty(S^* X^\circ))$ gives rise to a $\psi$do projection
$\Pi_+$ on $X$ which is in 
fact the truncation of a $\psi$do projection on $\wt X$. This will
allow us to calculate the residue of $\Pi_+$ from the residue of a
projection on the closed manifold $\wt X$.

\begin{lemma} \label{lem:proj}
Let $p(x,\xi) \in \IP_m(C_c^\infty(S^* X^\circ)^\sim)$. There
is a zero-order $\psi$do projection $\Pi$ acting on
$C^\infty(X,\C^m)$, such that its symbol is constant on $\widetilde X
\setminus X$, its truncation $\Pi_+$ is an idempotent in $\M_m(\mathscr
I^\sim)$, and 
\begin{equation}
  \label{eq:sigmastar}
\sigma_* q_* ( [ \Pi_+ ]_{0} ) = [ p ]_0.
\end{equation}
\end{lemma}

\begin{proof}
By definition of the unitization of $C_c^\infty(S^* X^\circ)$, we can
write $p$ as a sum
\begin{equation}
  \label{eq:pab}
  p(x,\xi) = \alpha(x,\xi) + \beta,
\end{equation}
with $\alpha \in \M_m(C_c^\infty(S^* X^\circ))$ and $\beta \in
\M_m(\C)$. Note that $\beta$ itself is idempotent, since $p =
\beta$ outside the support of $\alpha$.

We extend $\alpha$ by zero to obtain a smooth function on the closed manifold
$S^* \widetilde X$ denoted $\widetilde \alpha(x,\xi)$. We 
get a $\psi$do symbol (also denoted $\wt \alpha(x,\xi)$) of order
zero on $\widetilde X$ by requiring $\wt\alpha$ to be homogeneous
of degree zero in $\xi$. Let $\wt p(x,\xi) = \wt\alpha(x,\xi) +
\beta$.

We now have an idempotent $\psi$do-symbol $\wt p$ on $\wt X$; we then
construct a $\psi$do projection on $\wt X$ that has $\wt p$ as its
principal symbol. 

In \cite[Chapter 3]{G-zeta}, Grubb constructed an operator that, for a
suitable choice of atlas on the manifold, carries over to the
Euclidean Laplacian 
in each chart, modulo smoothing operators. Hence, choose that
particular atlas on 
$\widetilde X$ and let $D$ denote this particular operator, i.e., with
scalar symbol $d(x,\xi) = |\xi|^2$. Define the auxiliary second order
$\psi$do $C = \OP(c(x,\xi))$, with symbol $c(x,\xi)$ given in the
local coordinates of the specified charts as
\begin{equation}
  \label{eq:csymb}
c(x,\xi) = ( 2\widetilde p(x,\xi) - I ) d(x,\xi).  
\end{equation}
Since $\wt p$ is idempotent, the eigenvalues of
$2 \wt p - I$ are $\pm 1$, cf.\ \eqref{eq:app1}, so $C$ is an elliptic
second order operator and $c(x,\xi) - \lambda$ is parameter-elliptic
for $\lambda$ on each ray in $\mathbb C\setminus \mathbb R$.

Then we can define the sectorial projection, cf.\ \cite{Po06},
\cite{GG07}, $\Pi = \Pi_{\theta,\varphi}(C)$ with angles $\theta =
-\frac \pi2$, $\varphi = \frac\pi2$,
\begin{equation}
  \label{eq:Pidef}
\Pi = \frac{i}{2\pi}
\int_{\Gamma_{\theta,\varphi}} \lambda^{-1} C
(C-\lambda)^{-1}\, d\lambda.  
\end{equation}
$\Pi$ is a $\psi$do projection \cite{Po06} on $\wt X$ with symbol
$\pi$ given in local coordinates by
\begin{equation}
  \label{eq:pisymb}
\pi(x,\xi) = \frac{i}{2\pi} \int_{\mathcal C(x,\xi)} q(x,\xi,\lambda)
\, d\lambda,  
\end{equation}
where $q(x,\xi,\lambda)$ is the symbol with parameter for a
parametrix of $c(x,\xi) - \lambda$, and $\mathcal C(x,\xi)$ is a
closed curve encircling the eigenvalues of $c_2(x,\xi)$ --- the principal
symbol of $C$ --- in the $\{ \Re z > 0 \}$ half-plane.

The eigenvalues of $c_2(x,\xi) = (2\wt p(x,\xi) - I)
|\xi|^2$ are $\pm |\xi|^2$, so we can choose $\mathcal C(x,\xi)$ as the
boundary of a small ball $B(|\xi|^2,r)$ around $+|\xi|^2$.

Hence, the principal symbol of $\pi(x,\xi)$ is
\begin{align}
\pi_0(x,\xi) & = \frac{i}{2\pi} \int_{\mathcal C(x,\xi)}
q_{-2}(x,\xi,\lambda) \, d\lambda \notag \\ 
& = \frac{i}{2\pi} \int_{\partial B(|\xi|^2,r)} [ (2 \wt p(x,\xi) - I)
|\xi|^2 - \lambda]^{-1} \, d\lambda
= \wt p(x,\xi), 
\end{align}
according to Lemma \ref{lem:app1}.
So $\Pi$ is a $\psi$do projection with principal symbol
$\wt p(x,\xi)$, as desired. 

Observe that for $x$ outside the support of $\wt \alpha$, we have $c(x,\xi)
= (2\beta - I)|\xi|^2$ and 
$q(x,\xi,\lambda) = q_{-2}(x,\xi,\lambda) = ( (2\beta - I) |\xi|^2 -
\lambda)^{-1}$ so $\pi(x,\xi) = \pi_0(x,\xi) = \beta$ there. (We cannot be
sure that the full symbol of $\pi$ equals $\wt p$ inside the support,
since coordinate-dependence will in general influence the lower order terms of
the parametrix.) In particular,
$\pi(x,\xi)$ is constant equal to $\beta$ for $x\in \widetilde
X\setminus X$.

Now consider the truncation $\Pi_+$. We have
\begin{equation}
  \label{eq:Piplus}
(\Pi_+)^2 = (\Pi^2)_+ - L(\Pi, \Pi) = \Pi_+ - L(\Pi,\Pi),
\end{equation}
where the singular Green operator $L(P,Q)$ is defined as $(PQ)_+ - P_+
Q_+$ for $\psi$dos $P$ and $Q$.
Since $\pi(x,\xi)$ equals the constant matrix $\beta$ in a
neighborhood of the boundary $\partial X$ it follows, cf.\
\cite[Theorem 2.7.5]{G-book}, that $L(\Pi,\Pi)=0$, so $(\Pi_+)^2 =
\Pi_+$.

Since the symbol of $\Pi - \beta$ is compactly supported within
$X^\circ$, we can write $\Pi_+ = \varphi P \psi + \beta$ for some
$\varphi, \psi, P$, as in \eqref{eq:Iideal}; hence $\Pi_+$ is in
$\M_m(\mathscr I^\sim)$. Technically, $\Pi_+$ lies in the algebra where
the boundary bundle $F$ is the zero-bundle, but inserting zeros into
$\Pi_+$'s matrix form will clearly allow us to augment it to the
present case with $F = \partial X \times \C$.

Finally we take a look at \eqref{eq:sigmastar}: Since $\Pi_+$ is an
idempotent in $\M_m(\mathscr I^\sim)$ it defines a $K_0$-class
$[\Pi_+]_0$ in $K_0(\mathscr I^\sim)$. Then $q_* [\Pi_+]_0$ defines
a class in $K_0(\mathfrak I/\mathfrak K^\sim)$, a class
defined solely by its principal symbol. Since the principal symbol is
exactly the idempotent $p(x,\xi)$ we obtain \eqref{eq:sigmastar} by
definition.
\end{proof}

\begin{thm} \label{thm:resvan}
The noncommutative residue of any projection in (the norm closure of) 
the Boutet de Monvel calculus is zero.
\end{thm}

\begin{proof}
As mentioned, it suffices to show that $\res_*$ vanishes on
$K_0(\mathscr A) \cong K_0(\mathfrak A)$. In turn, according to 
equation \eqref{eq:split1} and Lemma \ref{lem:resmK}, we only need to
show that $\wt\res_*$ vanishes on $K_0(\mathfrak I/\mathfrak K)$.

So let $\omega \in K_0(\mathfrak I/\mathfrak K)$. Employing
\eqref{eq:stpic1}, \eqref{eq:sigmaiso}, and \eqref{eq:injiso} we can find
$p, p'$ in $\IP_m(C_c^\infty(S^* X^\circ)^\sim)$ such that
\begin{equation}
\label{eq:sigmaomega}
  \sigma_* \omega = [p]_0 - [p']_0.
\end{equation}

Now, for $p$, $p'$ we use Lemma \ref{lem:proj} to find corresponding
$\psi$dos $\Pi$, $\Pi'$ with the specific properties mentioned
there. By \eqref{eq:sigmastar} and \eqref{eq:sigmaomega} we see that
\begin{equation}
  \label{eq:pp}
  q_*[\Pi_+]_0 - q_* [\Pi'_+]_0 = \sigma_*^{-1} \big( [p]_0 - [p']_0
  \big) = \omega.
\end{equation}
Using equation \eqref{eq:restilde} we now see that
\begin{equation}
  \label{eq:resPI}
  \wt\res_* \omega = \res_X (\Pi_+) - \res_X (\Pi'_+).
\end{equation}
Here
\begin{equation}
  \label{eq:resPplus}
  \res_X (\Pi_+) = \int_{X} \int_{S^*_x X} \tr
  \pi_{-n}(x,\xi) \dbar   S(\xi) dx. 
\end{equation}
By construction, $\pi(x,\xi)$ is constant equal to $\beta$ outside $X$;
in particular $\pi_{-n}(x,\xi)$ is zero for $x \in \wt X\setminus X$ and
therefore
\begin{equation}
  \label{eq:res12}
\int_{X} \int_{S^*_x X} \tr \pi_{-n}(x,\xi) \dbar
S(\xi) dx = \int_{\widetilde X} \int_{S^*_x \widetilde X}
\tr \pi_{-n}(x,\xi) \dbar S(\xi) dx.
\end{equation}
In other words
\begin{equation}
  \label{eq:resP}
\res_X ( \Pi_+ ) = \res_{\wt X}( \Pi ),
\end{equation}
where the latter is the noncommutative residue of a $\psi$do
projection on a closed manifold. It is well-known \cite{Wod82},
\cite{Wod84} that the latter always vanishes. Likewise
we obtain $\res_X (\Pi'_+) = 0$ and finally
\begin{equation}
  \label{eq:resomega0}
\widetilde \res_* \omega = 0  
\end{equation}
as desired.
\end{proof}

In \cite{GG07}, it was an open question whether the residue is zero on
a sectorial projection for a boundary value 
problem. This theorem answers that question in the positive for the
cases where the sectorial projection lies in the $C^*$-closure of
$\mathscr A$.

It is not, at this time, clear for which boundary value problems this
is true. We showed in \cite{GG07} that there certainly are boundary
value problems where the sectorial projection is not in $\mathscr A$;
whether or not they lie in $\mathfrak A$ is something we intend to return
to in a future work.

\setcounter{section}{0}
\def\thesection{\Alph{section}}
\section{Appendix}

\begin{lemma}
\label{lem:app1}
Let $M \in \IP_m(\C)$. Let $d>0$ and let $\partial B(d, r)$
denote the closed curve in the complex plane along the boundary of the
ball with center $d$ and radius $0<r<d$. Then
\begin{equation}
\label{eq:app0}
  \frac{i}{2\pi} \int_{\partial B(d,r)} [ (2 M - I) d - \lambda ]^{-1}
  d\lambda 
  = M.
\end{equation}
\end{lemma}

\begin{proof}
A direct computation shows that, for $\lambda \neq \pm d$,
\begin{equation}
  \label{eq:app1}
  [ (2 M - I)d - \lambda ]^{-1} = \frac{M}{d-\lambda} -
  \frac{I-M}{d+\lambda}.
\end{equation}
The result in \eqref{eq:app0} then follows from the residue theorem.
\end{proof}

\section*{Acknowledgements}

The author is grateful to Gerd Grubb and Ryszard Nest for several
helpful discussions.

\end{document}